\documentclass[11pt]{article}
\usepackage[margin=1in]{geometry}
\usepackage{amsmath,amssymb,amsthm,mathtools}
\usepackage{booktabs}
\usepackage{microtype}
\usepackage{enumitem}
\usepackage{xcolor}
\usepackage{array}
\usepackage[ruled,vlined,algo2e]{algorithm2e}
\newcommand{\algcomment}[1]{\hspace{0.45em}{\color{gray}\footnotesize\normalfont/* #1 */}}
\usepackage{float}
\usepackage{cite}
\usepackage[hidelinks]{hyperref}
\usepackage[nameinlink,noabbrev]{cleveref}
\crefname{equation}{Eq.}{Eqs.}
\Crefname{equation}{Eq.}{Eqs.}
\let\parentheticaleqref\eqref
\renewcommand{\eqref}[1]{Eq.~\parentheticaleqref{#1}}

\newtheorem{theorem}{Theorem}[section]
\newtheorem{proposition}[theorem]{Proposition}
\newtheorem{lemma}[theorem]{Lemma}
\newtheorem{corollary}[theorem]{Corollary}
\newtheorem{remark}[theorem]{Remark}
\newtheorem{definition}[theorem]{Definition}
\newcommand{\T}{\mathbb T}

\newcommand{\ip}[2]{\left\langle #1,#2\right\rangle}
\newcommand{\norm}[1]{\left\lVert #1\right\rVert}
\newcommand{\TV}{\operatorname{TV}}

\newcommand{\proofstep}[1]{\par\medskip\noindent\textit{#1}\enspace}

\title{Successive Schur--Riesz Analysis for Approximation\\
\large Quotient Stability, Exact Enrichment, and Numerical Realizations}
\author{Matthew Dixon}
\date{August 2026}

\begin{document}
\maketitle

\begin{abstract}
Many approximation methods enlarge a trial space by adjoining function blocks
generated by different operators.  Exact redundancy and strong cross-level
interaction can make coefficients nonunique and render pairwise or
diagonal-dominance tests needlessly pessimistic.  For
\(V_m=\sum_{\ell\leq m}S_\ell(E_\ell)\) in a Hilbert space $\mathcal H$, we quotient
coefficients representing the same function and control successive orthogonal
innovations to obtain Riesz bounds independent of \(m\).  The setting includes
factored operators \(S_\ell=T_\ell\circ\cdots\circ T_1:E_\ell\to\mathcal H\),
with compatible intermediate spaces, but the theorem allows arbitrary bounded
\(S_\ell\).  The same constants control approximation, truncation,
perturbation, and levelwise error.  A block Schur complement identifies the
intrinsic new dimension and gives the exact reduction in squared
best-approximation error, leading to a constructive enrichment procedure.
Nonstationary and lifted examples give positive intrinsic bounds where
diagonal-dominance estimates are negative or labelled Gram matrices are
singular; adaptive and recycled-subspace calculations illustrate the distinct
roles of representation stability and application-specific utility.

\end{abstract}

\section{Introduction}

Many approximation spaces are formed by successively adjoining subspaces
spanned by functions produced by different operators.  Examples include
multiscale and subdivision constructions, dynamical sampling, adaptive
wavelets, compatible finite elements, operator-adapted bases, and successive
network layers
\cite{AldroubiEtAl2017,CohenDahmenDeVore2001,Cotter2023,DynLevin2002,
OwhadiScovel2019,DeVoreHaninPetrova2021,KidgerLyons2020}.  A ubiquitous
challenge is that coefficients may be nonunique while later subspaces interact
strongly with earlier ones.  Pairwise coherence and diagonal dominance can
then confuse exact redundancy with genuine instability.

\paragraph{Problem statement.}
If the admitted functions include $f$ and $2f$, the coefficient pairs $(1,0)$
and $(-1,1)$ represent the same approximant.  This exact nonuniqueness is
harmless.  After it is removed, however, genuinely different coefficient
classes should neither nearly cancel nor become progressively ill-conditioned
as more blocks are added.  Testing the complete Gram operator after every
enlargement detects failure but does not attribute it locally.  We therefore
seek an incremental criterion that removes null coefficients, isolates what a
new block contributes beyond the preceding space, controls the accumulation of
accepted blocks, and separates stable novelty from usefulness for a target.

\subsection{Literature review}
\label{sec:literature}

Riesz-basis and frame theory provides the natural language for stable linear
representation: coefficient perturbations are controlled in the represented
function norm, while frames permit redundancy \cite{Christensen2016}.  Fusion
frames extend this viewpoint from individual functions to subspaces, and
operator representations expose their block structure
\cite{BalazsEtAl2020,CasazzaKutyniokLi2008}.  These theories characterize a
prescribed family or decomposition, but do not by themselves give an
incremental criterion for successively admitted singular blocks with constants
uniform in the number of levels.

Block elimination addresses a different part of the problem.  Generalized
Schur complements measure a block after the preceding blocks have been removed
\cite{TarcsayTitkos2020}, whereas block Gershgorin estimates control raw
cross-interactions \cite{FeingoldVarga1962}.  The latter may give a negative
lower estimate for a stable represented family.  A one-step Schur complement,
however, does not establish that local elimination removes exactly the same
redundancy as the global synthesis quotient, nor that successive acceptable
steps remain uniformly stable.  That requires a compatible sequential
factorization.

Multiscale approximation shows why singular labelled descriptions should not
simply be excluded.  Tight and dual frames use redundancy to retain
localization, symmetry, or perfect reconstruction
\cite{ChuiHeStockler2004,Han2010}; lifting constructs new multiresolution
families from local updates \cite{CohenDyn1996,DaubechiesSweldens1998,Sweldens1998};
and nonstationary or lifted multiwavelets require stability analysis beyond a
single stationary Gram matrix \cite{MaesBultheel2008}.  Related hierarchical
principles appear in compatible finite elements, operator-adapted wavelets, and
Galerkin variational integrators
\cite{Cotter2023,HallLeok2015,LeimkuhlerReich2004,OwhadiScovel2019}.  These
constructions motivate structured generating blocks; the present question is
how to test their intrinsic stability after exact cross-level relations have
been removed.

Approximation theory supplies the complementary notion of target utility.
Nonlinear and greedy approximation measure error against effective degrees of
freedom and select useful directions
\cite{DeVore1998,Temlyakov1999,Temlyakov2008}; reduced-basis and adaptive
wavelet methods connect residual selection to convergence and approximation
classes
\cite{CohenDahmenDeVore2001,CohenDahmenDeVore2003,
DeVorePetrovaWojtaszczyk2013}.  Augmented and recycled Krylov methods similarly
reuse subspaces to accelerate related linear systems or reduced models
\cite{Bai2002,ParksEtAl2006,SoodhalterEtAl2020}.  These methods evaluate
usefulness, but do not generally begin by quotienting a redundant candidate
block and measuring the stability of its intrinsic innovation.  The missing
link is therefore a local object that handles redundancy and stability while
also attributing approximation gain.

\paragraph{Proposed approach and novelty.}
We prove three linked results.  First, the global synthesis quotient is
compatible with successive orthogonal innovations, including singular labelled
blocks.  Second, a block $LDL^*$ factorization converts local innovation and
elimination estimates into depth-uniform Riesz bounds, with approximation,
truncation, perturbation, and levelwise-error consequences.  Third, the same
Schur residual identifies intrinsic new dimension and the exact reduction in
squared best-approximation error.  Thus the object that removes redundancy also
measures stability, attributes gain, and supports a constructive enrichment
rule.  The theorem concerns arbitrary bounded generating operators; functional
composition, Galerkin hierarchies, and multilevel constructions are important
instances rather than assumptions.

\paragraph{Overview.}
Section~\ref{sec:riesz} develops the quotient--innovation theory,
Section~\ref{sec:algorithm} turns it into a numerical enrichment procedure, and
Section~\ref{sec:numerics} tests its stability and gain claims in nonstationary,
lifted, adaptive, and recycled-subspace examples.  Secondary results and audit
details are placed in the appendices.

\section{Successive quotient--innovation Riesz analysis}
\label{sec:riesz}

We derive stability bounds for successively enlarged approximation spaces that
are uniform in the number of levels and intrinsic under coefficient
nonuniqueness.  The construction separates exact redundancy from
near-dependence, isolates the contribution of each new block, and controls the
accumulation of cross-level interactions.

\subsection{Preliminaries}

Let $S_\ell:E_\ell\to\mathcal H$ be arbitrary bounded generating operators,
let
\[
 E^{(m)}=\bigoplus_{\ell=1}^mE_\ell,
 \quad
 S^{(m)}c=\sum_{\ell=1}^mS_\ell c_\ell,
 \quad
 V_m=S^{(m)}(E^{(m)}),
\]
where $\mathcal H$ is the Hilbert space in which approximation error is
measured and $E_\ell$ contains the coefficients introduced at level $\ell$.
The map $S_\ell$ converts those coefficients into functions.  Hence
\[
 V_m=\left\{\sum_{\ell=1}^mS_\ell c_\ell:c_\ell\in E_\ell\right\}
\]
is the linear approximation space generated through level $m$, and
$V_1\subseteq V_2\subseteq\cdots$.  The superscript $(m)$ denotes the
assembly of the first $m$ coefficient blocks, whereas the subscript on $V_m$
labels the resulting nested space.

Different coefficient vectors may determine the same element of $V_m$.
After identifying these vectors, the contribution of the newly adjoined range
is measured relative to $V_{m-1}$.

\subsection{Intrinsic coefficients and successive orthogonal residuals}
\label{sec:block}

Two reductions are required before a stability theorem can be stated.
Quotienting comes first: it removes coefficient changes that produce no change
in the represented function.  Orthogonal residualization comes second: it asks
what remains of level $\ell$ after everything already representable at earlier
levels has been removed.

\paragraph{Removing exact redundancy.}
The first issue appears before one asks whether the growing spaces are well
conditioned.  Two coefficient vectors can differ while representing the same
function.  The Gram operator
$G^{(m)}=(S^{(m)})^*S^{(m)}$ records the norm of every represented function:
$\norm{S^{(m)}c}^2=\ip{G^{(m)}c}{c}$.  Because $G^{(m)}$ is singular whenever
coefficients are exactly redundant, we work on the intrinsic coefficient
space
\[
 \widehat E^{(m)}=E^{(m)}/\ker S^{(m)},\quad
 \norm{[c]}=\inf_{z\in\ker S^{(m)}}\norm{c+z}.
\]
Here
\[
 \ker S^{(m)}=\{c\in E^{(m)}:S^{(m)}c=0\}
\]
is the subspace of coefficient changes that leave the represented function
unchanged.  Thus $[c]$ contains all coefficient vectors representing the same
function as $c$, and its norm is the norm of the shortest such representative.
Nothing is removed from $V_m$; only duplicate coefficient descriptions are
identified.

\paragraph{Isolating what the new level contributes.}
After exact redundancy has been removed, the next question is geometric.  A
new block $S_\ell(E_\ell)$ may point partly into the preceding space
$V_{\ell-1}$ and partly outside it.  Let $P_{\ell-1}$ be the orthogonal
projection onto $V_{\ell-1}$ and define
\[
 R_\ell=(I-P_{\ell-1})S_\ell.
\]
The vector $R_\ell b$ is exactly the part of $S_\ell b$ that cannot already
be represented at earlier levels.  Its Gram operator is
$D_\ell=R_\ell^*R_\ell$, and the projection theorem gives
\begin{equation}
 \ip{D_\ell b}{b}
 =\inf_{a\in E^{(\ell-1)}}
   \norm{S_\ell b-S^{(\ell-1)}a}_{\mathcal H}^2.
 \label{eq:innovationdistance}
\end{equation}
Thus $D_\ell$ measures the squared distance of the new contribution from
everything admitted before.

For computation, assume that the preceding synthesis operator has closed
range, so its Moore--Penrose inverse is bounded, and write
$G_{<\ell,<\ell}$ for the Gram operator of the preceding levels.  Eliminating
the best preceding coefficient in \eqref{eq:innovationdistance} gives the
generalized Schur complement
\begin{equation}
 D_\ell=G_{\ell\ell}
 -G_{\ell,<\ell}G_{<\ell,<\ell}^{\dagger}G_{<\ell,\ell},
 \label{eq:generalizedinnovation}
\end{equation}
where $\dagger$ is the Moore--Penrose inverse on the closed range.  The Schur
complement is therefore the algebraic form of orthogonally removing the
preceding approximation space, not a separate construction.
Thus $\ker D_\ell$ is precisely the part of the new generating family already represented
by the preceding span.  We quotient it out and retain
$\widehat E_\ell=E_\ell/\ker D_\ell$.

\paragraph{Reconciling local and global descriptions.}
We now have two reductions: one global quotient for all coefficients through
level $m$, and one local quotient at each level.  The next lemma proves that
they discard exactly the same redundancy.  This is the bridge that permits
local innovation estimates to control the complete growing space.

\begin{lemma}[Compatibility of global and successive quotients]
\label{lem:compatibility}
Assume $V_m=S^{(m)}(E^{(m)})$ is closed for every $m$, with
$V_0=\{0\}$, and put $W_m=V_m\ominus V_{m-1}$.  Let $P_{m-1}$ be
the orthogonal projection onto $V_{m-1}$ and define
\[
 R_m=(I-P_{m-1})S_m:E_m\to W_m.
\]
Then $D_m=R_m^*R_m$ equals \eqref{eq:generalizedinnovation}, and $R_m$
induces a bounded Hilbert-space isomorphism
$\widehat R_m:E_m/\ker D_m\to W_m$.  In particular,
\begin{equation}
 \ker D_m=\ker R_m
 =\{b\in E_m:S_mb\in V_{m-1}\}.
 \label{eq:innovationkernel}
\end{equation}
Moreover,
\[
 V_m=W_1\mathbin{\mathop\oplus\limits^\perp}\cdots
 \mathbin{\mathop\oplus\limits^\perp}W_m.
\]
Consequently there is a unique bounded isomorphism
\[
 C_m:E^{(m)}/\ker S^{(m)}
 \longrightarrow \bigoplus_{\ell=1}^m E_\ell/\ker D_\ell
\]
that assigns to each represented function its successive orthogonal-residual
coordinates.  It is triangular with identity diagonal blocks.  Hence, with
$L_m=C_m^*$,
\[
 \widehat G^{(m)}=L_m\mathcal D_mL_m^*,
 \quad \mathcal D_m=\operatorname{diag}(D_1,\ldots,D_m).
\]
The construction is compatible under truncation: the leading $k$ blocks of
$C_m$ agree with $C_k$ for $k\leq m$.
\end{lemma}

\begin{proof}
\proofstep{Step 1: identify the residual Gram operator.}
For $b\in E_m$, the projection theorem gives
\[
 \norm{R_mb}^2
 =\inf_{a\in E^{(m-1)}}\norm{S_mb-S^{(m-1)}a}^2.
\]
Expanding the least-squares normal equation with the Moore--Penrose solution
gives \eqref{eq:generalizedinnovation}; hence $D_m=R_m^*R_m$ and
$\ker D_m=\ker R_m$.  Since $R_mb=0$ exactly when
$S_mb=P_{m-1}S_mb\in V_{m-1}$, this also proves
\eqref{eq:innovationkernel}.  Notice that the Moore--Penrose inverse used
here is bounded: closedness of $V_{m-1}=S^{(m-1)}(E^{(m-1)})$
implies closedness of the range of its Gram operator.

\proofstep{Step 2: quotient exactly.}
The map $R_m$ is onto $W_m$.  Indeed, if $w\in W_m\subset V_m$, write
$w=v+S_mb$ with $v\in V_{m-1}$; applying $I-P_{m-1}$ gives $w=R_mb$.
Thus the induced map $\widehat R_m:E_m/\ker R_m\to W_m$ is bijective.
Both spaces are Hilbert spaces and the map is bounded, so the bounded inverse
theorem makes it a Hilbert-space isomorphism.  Because
$\ker R_m=\ker D_m$, this is exactly the asserted local quotient, including
when $D_m$ is singular on the labelled coefficient space.

\proofstep{Step 3: decompose the nested ranges.}
Since $V_m=V_{m-1}\oplus^\perp W_m$, induction gives the displayed
orthogonal sum.  The map
\[
 (b_1,\ldots,b_m)\longmapsto
 \sum_{\ell=1}^m\widehat R_\ell b_\ell
\]
is therefore a bounded isomorphism from the direct sum of the local quotient
spaces onto $V_m$; its Gram operator is
$\mathcal D_m=\operatorname{diag}(D_1,\ldots,D_m)$.

\proofstep{Step 4: identify the two quotient descriptions.}
Both the global quotient and the direct sum of local quotients map bijectively
and boundedly onto the same space $V_m$.  Composing one map with the inverse
of the other gives $C_m$.  This proves the central compatibility statement:
sequential elimination removes all and only the globally null coefficient
directions.  A block
introduced at level $\ell$ has its $W_\ell$ residual as diagonal coordinate
and projections only in $W_k$, $k<\ell$; this proves the triangular statement
and identity diagonal blocks.  Comparing the two Gram operators gives
\[
 (\widehat S^{(m)})^*\widehat S^{(m)}
 =C_m^*\mathcal D_mC_m,
\]
which is the claimed factorization with $L_m=C_m^*$.

\proofstep{Step 5: verify truncation compatibility.}
For $k\le m$, the spaces $V_j,W_j$, projections $P_j$, and residual maps
$R_j$ with $j\le k$ depend only on the first $k$ levels.  The leading
coordinates of $C_m$ therefore agree with $C_k$; no choice of representative
or generalized inverse can change them.
\end{proof}

The lemma supplies the required coordinates.  The local quotient spaces
contain only genuinely new directions, while the triangular factor records
how the original coefficient blocks are reconstructed from those directions.
The stability theorem now follows by placing uniform bounds on these two parts
of the factorization.

\begin{theorem}[Successive-Space Riesz Theorem]
\label{thm:quotientinnovation}
Use the canonical factorization of \Cref{lem:compatibility} and suppose
\begin{equation}
 \widehat G^{(m)}=L_m\mathcal D_mL_m^*,
 \quad \mathcal D_m=\operatorname{diag}(D_1,\ldots,D_m),
 \label{eq:ldl}
\end{equation}
with $L_m=C_m^*$ on the successive residual quotients.
Assume, uniformly in $m$ and $\ell$,
\begin{equation}
 \delta I\preceq D_\ell\preceq\Delta I,\quad
 \norm{L_m}\leq M,\quad \norm{L_m^{-1}}\leq K
 \label{eq:innovationassumptions}
\end{equation}
for constants $0<\delta\leq\Delta<\infty$ and $M,K<\infty$.
Here $D_\ell$ acts on the quotient $E_\ell/\ker D_\ell$, canonically
realized as $(\ker D_\ell)^\perp$, and $I$ is the identity there.
Then
\begin{equation}
 \frac{\delta}{K^2}\norm{[c]}^2
 \leq\norm{S^{(m)}[c]}_{\mathcal H}^2
 \leq\Delta M^2\norm{[c]}^2,\quad m\geq1.
 \label{eq:quotientriesz}
\end{equation}
The coefficient-to-function maps extend consistently to countably many
levels.  On the resulting intrinsic coefficient space the map remains bounded
below by $\sqrt\delta/K$, its image is closed, and every represented function
has a unique minimum-norm intrinsic coefficient.
\end{theorem}

\begin{proof}
\proofstep{Step 1: remove null directions.}
The induced synthesis map on $(\ker S^{(m)})^\perp$, the canonical realization
of $\widehat E^{(m)}$, is injective and has Gram operator
$\widehat G^{(m)}$.

\proofstep{Step 2: pass to orthogonal residual coordinates.}
By \eqref{eq:ldl},
\[
 \norm{S^{(m)}c}^2
 =\ip{\mathcal D_mL_m^*c}{L_m^*c}.
\]
Consequently
$\delta\norm{L_m^*c}^2\leq\norm{S^{(m)}c}^2
\leq\Delta\norm{L_m^*c}^2$.

\proofstep{Step 3: control accumulated elimination.}
The bounds on $L_m$ and $L_m^{-1}$ give
$K^{-1}\norm c\leq\norm{L_m^*c}\leq M\norm c$, proving
\eqref{eq:quotientriesz}.

\proofstep{Step 4: pass to infinitely many blocks.}
The constants are independent of $m$.  Synthesis therefore extends by density
from finitely supported quotient sequences, and the lower bound passes to the
limit.  A bounded-below operator has closed range; quotienting gives the
unique minimum-norm representative.
\end{proof}

\paragraph{Discussion.}
The theorem permits large pairwise cross-block correlations because it tests a
new block against the \emph{whole} preceding space rather than summing raw
pairwise interactions.  Exact copies disappear in the quotient; near copies
appear as small eigenvalues of $D_\ell$; and harmful accumulation across many
otherwise acceptable levels appears through $L_m$.  These are mathematically
distinct failure modes, although a test on the full labelled Gram matrix
conflates them.

In frame-theoretic terms, this is a Riesz estimate on the intrinsic
coefficient space.  Its additional content is that the estimate is assembled
successively from quotient-compatible Schur innovations
\cite{BalazsEtAl2020,Christensen2016,TarcsayTitkos2020}.

A positive innovation bound establishes stable novelty, not utility for a
particular target.  Moreover, without locality, decay, recurrence, or another
structural restriction, controlling $L_m$ may be as difficult as conditioning
the complete Gram operator.  The following corollaries give the approximation
consequences under the stated hypotheses; \Cref{sec:checkable} gives structural
conditions under which those hypotheses are local.

\begin{corollary}[Approximation, truncation, and error localization]
\label{cor:approximation}
Under \Cref{thm:quotientinnovation}, set $A=\delta/K^2$,
$B=\Delta M^2$, $V_m=S^{(m)}(E^{(m)})$, and
$V=\overline{\bigcup_mV_m}$.  If $P_m$ and $P_V$ are the corresponding
orthogonal projections, then for every $f\in\mathcal H$:
\begin{enumerate}[label=(\roman*),leftmargin=2em]
\item $P_mf\to P_Vf$, and the minimum-norm quotient coefficient $c_m$ of
      $P_mf$ satisfies $\norm{c_m}\leq A^{-1/2}\norm{P_mf}$;
\item if $c$ synthesizes $P_Vf$ and $Q_m$ truncates its quotient coordinates,
      then
      $\norm{P_Vf-SQ_mc}\leq\sqrt B\norm{(I-Q_m)c}$;
\item with $P_0=0$,
\[
 \norm{f-P_mf}^2=\norm{f-P_Vf}^2
 +\sum_{\ell>m}\norm{P_\ell f-P_{\ell-1}f}^2.
\]
\end{enumerate}
\end{corollary}

\begin{proof}
Nested closed subspaces have strongly convergent orthogonal projections.  The
coefficient estimate and truncation estimate follow from the lower and upper
Riesz bounds.  The increments
$P_\ell f-P_{\ell-1}f\in V_\ell\ominus V_{\ell-1}$ are mutually orthogonal,
so the last identity is Pythagoras followed by passage to the limit.
\end{proof}

The corollary gives the approximation consequences.  Uniform Riesz bounds do
more than stabilize coefficients: they ensure convergence of the nested best
approximations, bound the error caused by discarding later coefficient blocks,
and decompose the remaining projection error into orthogonal levelwise
increments.  The identity is exact in the Hilbert norm; application-specific
estimators are still needed when that norm is not directly observable.

\begin{corollary}[Perturbation stability]
\label{cor:perturbation}
If $S$ has quotient-Riesz bounds $A,B$ on a fixed reduced coefficient space
and $\norm{\widetilde S-S}\leq\varepsilon<\sqrt A$, then
\[
 (\sqrt A-\varepsilon)^2\norm c^2
 \leq\norm{\widetilde Sc}^2
 \leq(\sqrt B+\varepsilon)^2\norm c^2.
\]
\end{corollary}

\begin{proof}
Apply the triangle and reverse-triangle inequalities to
$\widetilde Sc=Sc+(\widetilde S-S)c$.
\end{proof}

Perturbation stability supplies a margin interpretation of the lower Riesz
constant: changes smaller than $\sqrt A$ cannot destroy injectivity on the
reduced coefficient space.  It is deliberately a representation statement,
not a convergence theorem for a nonlinear optimizer.

\subsection{When the hypotheses can be checked locally}
\label{sec:checkable}

The factor $L_m$ should not be estimated by diagonalizing the full Gram
matrix.  Write $L_m=I+N_m$ and let
$n_{\ell k}=\norm{(N_m)_{\ell k}}$ for $k<\ell$.  These blocks are sequential
regression coefficients: they describe how a newly introduced generating
family projects onto earlier
orthogonal residual spaces after all intervening redundancy has been eliminated.

\begin{proposition}[Summable regression criterion]
\label{prop:regressioncriterion}
Suppose, uniformly in $m$,
\[
 R_N=\sup_\ell\sum_{k<\ell}n_{\ell k}<\infty,
 \quad
 C_N=\sup_k\sum_{\ell>k}n_{\ell k}<\infty,
 \quad q=\sqrt{R_NC_N}<1.
\]
Then
\[
 \norm{N_m}\leq q,\quad
 \norm{L_m}\leq1+q,\quad
 \norm{L_m^{-1}}\leq\frac1{1-q}.
\]
Thus \Cref{thm:quotientinnovation} is activated by blockwise regression
estimates without requiring small raw cross-Gram row sums.
\end{proposition}

\begin{proof}
The block Schur test gives
$\norm{N_m}^2\leq R_NC_N=q^2$.  Since $q<1$, the Neumann series
$L_m^{-1}=\sum_{j\geq0}(-N_m)^j$ converges in operator norm and has norm at
most $(1-q)^{-1}$; the upper bound for $L_m$ is the triangle inequality.
\end{proof}

\begin{proposition}[Two further structural routes to uniform elimination]
\label{prop:structuralelimination}
Besides the small-regression condition of
\Cref{prop:regressioncriterion}, either of the following hypotheses gives
bounds on \(L_m\) and \(L_m^{-1}\) without spectral analysis of the complete
Gram operator.
\begin{enumerate}[label=(\roman*),leftmargin=2em]
\item \emph{Factorized local updates.} Suppose
\[
 L_m=\prod_{j=1}^{J_m}(I+K_{m,j}),\qquad \norm{K_{m,j}}<1,
\]
with the order fixed by the elimination, and suppose uniformly in \(m\)
\[
 \sum_{j=1}^{J_m}\norm{K_{m,j}}\leq C_+,\qquad
 \sum_{j=1}^{J_m}-\log(1-\norm{K_{m,j}})\leq C_-.
\]
Then
\[
 \norm{L_m}\leq e^{C_+},\qquad
 \norm{L_m^{-1}}\leq e^{C_-}.
\]
\item \emph{Causal operator symbol.} Suppose all quotient blocks are copies
of one finite-dimensional space \(E\), and the infinite elimination operator
on \(\ell^2(\mathbb N;E)\) is the causal convolution
\[
 (Lc)_\ell=\sum_{j=0}^{\ell}A_jc_{\ell-j},\qquad A_0=I,
 \quad \alpha:=\sum_{j\geq0}\norm{A_j}<\infty.
\]
If the analytic matrix symbol
\(\mathcal L(z)=\sum_{j\geq0}A_jz^j\) is invertible for
\(\lvert z\rvert\leq1\), and
\(\mathcal L(z)^{-1}=\sum_{j\geq0}B_jz^j\) satisfies
\(\beta:=\sum_{j\geq0}\norm{B_j}<\infty\), then every causal finite section
satisfies
\[
 \norm{L_m}\leq\alpha,\qquad \norm{L_m^{-1}}\leq\beta.
\]
For a matrix polynomial symbol, nonvanishing of
\(\det\mathcal L(z)\) on the closed unit disk implies the required absolute
summability of the inverse coefficients.
\end{enumerate}
\end{proposition}

The proof uses product bounds and Young's convolution inequality and is given
in \Cref{app:structuralproofs}.

Two useful specializations are immediate.  If
$n_{\ell k}\leq C_0a_{\ell-k}$ for a nonnegative summable sequence and
$C_0\sum_{j\geq1}a_j<1$, then the proposition applies.  This includes banded
regressions and Jaffard-type algebraic decay
$a_j=(1+j)^{-s}$, $s>1$.  These are localization conditions on
\emph{conditional regressions}, not diagonal dominance of the original Gram
operator.

Principal angles control the other factor.  If
$G_{\ell\ell}\succeq aI$ and every vector in
$S_\ell(E_\ell)$ makes angle at least $\theta_0>0$ with
$V_{\ell-1}$ after quotienting their intersection, then
\begin{equation}
 D_\ell\succeq a\sin^2(\theta_0)I.
 \label{eq:angleinnovation}
\end{equation}
Indeed, \eqref{eq:innovationdistance} is the squared distance to
$V_{\ell-1}$.  Combining \eqref{eq:angleinnovation} with
\Cref{prop:regressioncriterion} makes every constant operational.

The causal-symbol route in \Cref{prop:structuralelimination} is stronger than
testing each finite section separately: one fixed analytic symbol controls all
truncation levels. The same bounds survive an asymptotically stationary
operator perturbation \(E\) when \(\beta\norm E<1\), with inverse bound
\(\beta/(1-\beta\norm E)\).

\paragraph{When these assumptions are genuinely easier.}
The elimination hypothesis is not automatically weaker than conditioning the
full Gram operator. Once
\(\delta I\preceq\mathcal D_m\preceq\Delta I\), the identity
\(\widehat G^{(m)}=L_m\mathcal D_mL_m^*\) shows that uniform Gram conditioning
and uniform two-sided control of \(L_m\) are quantitatively equivalent up to
\(\delta,\Delta\). No local criterion can therefore simplify an arbitrary
dense, unstructured family.

The advantage appears when the generating mechanism supplies structure before
the full Gram matrix is formed. Localized conditional regressions use row and
column sums independent of \(m\); lifting or multilevel constructions expose
products of local update operators; stationary recurrences reduce all finite
sections to one fixed matrix symbol. In these cases the verification data have
bounded local size or form summable sequences, whereas the quotient Gram
dimension grows with \(m\) and its smallest eigenvalue would otherwise need to
be recomputed after every enlargement.

\begin{table}[H]
\centering
\small
\caption{Structural routes for verifying the elimination hypothesis. The
last row states the boundary of the method rather than a positive criterion.}
\label{tab:eliminationroutes}
\begin{tabular}{@{}p{0.22\textwidth}p{0.31\textwidth}p{0.37\textwidth}@{}}
\toprule
Available structure & Quantity checked & Why it avoids the growing Gram test\\
\midrule
Localized nonstationary blocks & Uniform row/column sums of conditional
regression norms & Only a fixed or summably decaying level neighborhood is
required\\
Lifting or local multilevel updates & Summable norms of elementary factors
and their inverse margins & Bounds compose from local updates already exposed
by the construction\\
Stationary or asymptotically stationary recurrence & One causal matrix symbol
and its inverse coefficients & A fixed finite-dimensional symbol controls all
leading truncations\\
Dense unstructured interactions & No smaller structural datum is available &
Uniform \(L_m\) control is essentially the original Gram-conditioning problem\\
\bottomrule
\end{tabular}
\end{table}

\section{Numerical approximation}
\label{sec:algorithm}

We now derive an enrichment method from the stability theory.  A well
conditioned candidate need not improve the current approximation.  The
residual Gram operator therefore measures stable innovation, while its action
on the target residual measures approximation gain.  Together they define an
incremental procedure for possibly redundant block families.

This connects the construction with greedy and adaptive approximation:
target residuals determine usefulness, but each candidate is first
reduced to its stable intrinsic innovation
\cite{CohenDahmenDeVore2003,DeVorePetrovaWojtaszczyk2013,Temlyakov2008}.

Let $F_q$ be a finite-dimensional candidate coefficient space
and let $T_q:F_q\to\mathcal H$ be bounded.  At the
current closed approximation space $V$ with orthogonal projector $P$, define
\begin{equation}
 R_q=(I-P)T_q,\qquad D_q=R_q^*R_q,
 \qquad b_q=R_q^*(f-Pf).
 \label{eq:algorithmaudit}
\end{equation}
The first two objects depend only on the candidate and the current space; the
covector $b_q$ also depends on the target $f$.  For a numerical innovation
threshold $\tau\geq0$, let $\Pi_{q,\tau}$ be the spectral projector of $D_q$
on $[\tau,\infty)$, with zero eigenvalues omitted when $\tau=0$, and let
$D_{q,\tau}^{\dagger}$ denote the inverse on this retained quotient range and
zero on its orthogonal complement.  Put
\begin{equation}
 \Gamma_q
 =\ip{b_q}{D_{q,\tau}^{\dagger}b_q},\qquad
 W_q=R_q\Pi_{q,\tau}(F_q).
 \label{eq:algorithmgain}
\end{equation}
Thus $\dim W_q$ is the intrinsic cost of the candidate block, not its number
of redundant labels.  A normalized permutation or splitting of generators
does not alter $W_q$ or $\Gamma_q$; more generally these objects are intrinsic
under an isometric change of quotient coordinates.

\begin{proposition}[Exact Schur--Riesz block-enrichment theorem]
\label{prop:innovationgain}
Let $P_{V+W_q}$ be the orthogonal projector onto $V+W_q$.  Then
\begin{equation}
 \norm{f-Pf}^2-\norm{f-P_{V+W_q}f}^2=\Gamma_q.
 \label{eq:exactgain}
\end{equation}
Moreover, every retained quotient direction has innovation energy at least
$\tau$, and its minimum-norm update satisfies
\begin{equation}
 \norm{D_{q,\tau}^{\dagger}b_q}
 \leq \tau^{-1}\norm{\Pi_{q,\tau}b_q}
 \quad(\tau>0).
 \label{eq:updatebound}
\end{equation}
\end{proposition}

\begin{proof}
Because $W_q\perp V$, the new component of the best approximation is the
projection of $r=f-Pf$ onto $W_q$.  In retained quotient coordinates that
projection is $R_qD_{q,\tau}^{\dagger}b_q$.  Its squared norm is
\[
 \ip{D_{q,\tau}^{\dagger}b_q}
     {D_qD_{q,\tau}^{\dagger}b_q}
 =\ip{b_q}{D_{q,\tau}^{\dagger}b_q}=\Gamma_q.
\]
Pythagoras gives \eqref{eq:exactgain}.  The spectral definition of
$\Pi_{q,\tau}$ gives the lower innovation bound, while the norm of the inverse
on that range is at most $\tau^{-1}$, proving \eqref{eq:updatebound}.
\end{proof}

The preceding construction can be expressed as a general numerical
approximation algorithm.  The pseudocode is given in \Cref{alg:ssr}.  Starting
from a closed trial space, it tests each proposed block only through its
component orthogonal to the current space, removes null and poorly conditioned
directions, and computes the resulting reduction in best-approximation error.
It then admits the highest-scoring block that preserves the prescribed Riesz
lower bound.

In the algorithm, $C(V)$ measures resource cost, $\varepsilon$ is the error
tolerance, $\eta$ is the minimum accepted gain, and $B$ is the resource budget.
For candidate $q$, let $A_q^{\mathrm{new}}$ denote the quotient--Riesz lower
bound after adjoining $W_q$; the prescribed stability threshold is
$A_{\min}>0$.  A typical score is
$s(q,W_q,\Gamma_q)=\Gamma_q/\dim W_q$.

\noindent
\begin{algorithm2e}[H]
\small
\caption{Successive Schur--Riesz approximation}
\label{alg:ssr}
\KwIn{$f$, initial space $V$, candidates $\mathcal Q(V)$, threshold
$\tau\geq0$, tolerances $(\varepsilon,\eta,A_{\min})$, budget $B$, cost $C$,
and score $s$.}
\KwOut{$V$ and its compatible quotient factorization $(L,D)$.}
$P\leftarrow P_V$; $r\leftarrow f-Pf$\;
\While{$\norm r>\varepsilon$ and $C(V)<B$}{
  $\mathcal A\leftarrow\varnothing$\;
  \ForEach{$q\in\mathcal Q(V)$}{
    $(D_q,b_q)\leftarrow$ \eqref{eq:algorithmaudit}
    \algcomment{conditional Gram data}\;
    $W_q\leftarrow R_q\Pi_{q,\tau}(F_q)$
    \algcomment{stable innovation}\;
    $\Gamma_q\leftarrow\ip{b_q}{D_{q,\tau}^{\dagger}b_q}$
    \algcomment{exact gain}\;
    \If{$W_q\neq\{0\}$, $\Gamma_q\geq\eta$, and
    $A_q^{\mathrm{new}}\geq A_{\min}$}{
      $\mathcal A\leftarrow\mathcal A\cup\{q\}$\;
    }
  }
  \eIf{$\mathcal A=\varnothing$}{
    \textbf{break}\;
  }{
    $q^*\leftarrow\arg\max_{q\in\mathcal A}s(q,W_q,\Gamma_q)$\;
    $V\leftarrow V+W_{q^*}$
    \algcomment{enrichment}\;
    Update $(P,r,L,D)$ and the bounds in
    \Cref{thm:quotientinnovation}\;
  }
}
\Return{$(V,L,D)$}\;
\end{algorithm2e}

Thus each iteration has three stages: compute the current residual, test the
conditional innovations supplied by the candidate blocks, and enlarge the
space by one stable block with sufficient exact gain.  The procedure terminates
when the residual tolerance is reached, the resource budget is exhausted, or
no candidate satisfies both the gain and stability thresholds.

Two realizations are natural.  Storing an orthonormal basis of the spaces
$W_q$ gives canonical numerical coordinates and unit local Gram blocks.
Retaining the original structured generators preserves locality,
compositional form, or physical interpretation; the quotient $LDL^*$ bounds
then control their coefficient stability.  In either realization,
\eqref{eq:exactgain} attributes the improvement to the accepted block exactly,
without a Taylor approximation.

Projection, pseudoinversion, Schur elimination, and greedy selection are
standard operations.  The new assertion is the joint one: for an arbitrary
redundant candidate block, the same quotient residual (i) identifies its
intrinsic new dimension, (ii) provides its local stability spectrum,
(iii) gives the exact target-error reduction \(\Gamma_q\), and (iv) enters the
updated depth-uniform Riesz bound.  Classical orthogonal greedy methods use
target correlation to select atoms, but do not by themselves provide this
successive quotient-stability certificate for singular structured blocks.
Conversely, a Riesz audit alone does not select a useful block.

Application-specific information still enters through the candidate blocks
and their costs.  A Galerkin method may supply locally marked enrichments,
while a recycled Krylov method may supply Ritz blocks.  In the latter case,
positive Schur innovation certifies a stable recycle space but does not by
itself certify iteration reduction; classical reduced-spectrum analysis is
still required.

Thus quotienting identifies equivalent functional directions, the Schur
spectrum measures the stability of the remaining innovation, and $\Gamma_q$
gives its exact reduction of the current best-approximation error.  The next
section tests these statements separately and jointly.

\section{Examples and numerical results}
\label{sec:numerics}

The aim of this section is not to model a complete real-world application or
to establish broad performance superiority.  Instead, deliberately controlled
examples isolate the principal mechanisms of the method and show what
additional information the quotient--innovation analysis provides.  The
examples are drawn from familiar settings---nonstationary recurrences, lifted
multiscale approximation, adaptive enrichment, and recycled Krylov
subspaces---so that the same analysis also illustrates how the method enters
established approximation and numerical procedures.

The examples test the principal assertions separately.  An alternating
recurrence and a level-dependent lifted family test the stability theorem where
diagonal dominance fails.  Adaptive selection on the lifted family tests the
exact block-enrichment identity.  A recycled-CG sequence then separates
stable quotient representation from downstream solver utility.  The first
two constructions are deterministic; the Krylov table aggregates random bulk
realizations while preserving an exactly specified spectral structure.

Accordingly, the experiments return to the three settings that motivated the
framework: failure of diagonal dominance, redundancy in lifted multiscale
families, and reuse of historical numerical subspaces
\cite{ChuiHeStockler2004,FeingoldVarga1962,ParksEtAl2006}.

\paragraph{Evaluation protocol.}
Each example separates quantities computed from the theory from quantities
used only for validation.  The former are obtained from local residual Gram
operators, elimination bounds, or the inputs to \Cref{alg:ssr}; the latter are
finite Gram eigenvalues, measured approximation errors, or solver iteration
counts.  No validation quantity is substituted into a theorem bound.
The following map records the precise correspondence.

\begin{center}
\centering
\scriptsize
\textbf{Theory-to-experiment map.}\quad ``Prediction'' denotes a quantity
computed without the independent check in the last column.\par\smallskip
\begin{tabular}{@{}p{0.19\textwidth}p{0.24\textwidth}p{0.27\textwidth}p{0.22\textwidth}@{}}
\toprule
Example & Inputs & Theoretical output or prediction & Independent check\\
\midrule
Alternating recurrence
& $(\rho_\ell)$, truncation $m$, target coefficients
& residual bounds $(\delta,\Delta)$, elimination bounds $(M,K)$, Riesz and
truncation bounds from \Cref{thm:quotientinnovation,cor:approximation}
& finite quotient-Gram eigenvalues and measured truncation error\\
Lifted families
& masks $(\alpha_j)$, parent map, depth $J$
& kernel dimension, innovation spectrum, quotient lower bound, reconstruction
and truncation bounds
& quotient-Gram eigenvalues, amplification, and function-space error\\
Adaptive paired lifting
& target $f$, $V_0$, candidate pool, spectral cutoff, score, and dimension budget
& \Cref{alg:ssr}: selected blocks, intrinsic ranks, exact gains, and spaces $V_B$
& observed squared-error decrease, split-label invariance, and error at budget $B$\\
Recycled CG
& $(A_t,b_t)$, labelled recycle block $U$, setup cost
& quotient lower bound and classical reduced-spectrum iteration prediction
& accepted/rejected blocks and measured matrix--vector products\\
\bottomrule
\end{tabular}
\end{center}

\subsection{An alternating level-dependent recurrence}
\label{sec:nonstationaryexample}

Level-dependent recurrences are elementary models for nonstationary refinement
and operator-generated sampling, where the generating rule changes with scale
or iteration \cite{AldroubiEtAl2017,CohenDyn1996,DynLevin2002}.  The first
experiment isolates stability without redundancy or target selection.  It
asks whether conditional innovations can certify a strongly interacting,
nonstationary family after a classical row-sum estimate has become vacuous.

Let $(e_\ell)_{\ell\geq1}$ be an orthonormal sequence in
$L^2(\nu)$ and define the recursively generated system
\begin{equation}
 h_1=e_1,\quad
 h_\ell=\rho_\ell h_{\ell-1}
 +\sqrt{1-\rho_\ell^2}\,e_\ell,\quad
 \rho_\ell=\begin{cases}0.8,&\ell\text{ even},\\0.6,&\ell\text{ odd}.
 \end{cases}
 \label{eq:alternatingrecursion}
\end{equation}
This is level dependent: the component map alternates with the index and the Gram
matrix is not Toeplitz.  It is a special case of
$T_\ell(z,x)=\rho_\ell z+\sqrt{1-\rho_\ell^2}e_\ell(x)$ with identity
observable.

\begin{proposition}[Strict gain over diagonal dominance]
\label{prop:alternatingexample}
For the system \eqref{eq:alternatingrecursion},
\[
 \ip{h_i}{h_j}=\prod_{k=i+1}^j\rho_k\quad(i<j),
 \quad D_1=1,\quad D_\ell=1-\rho_\ell^2.
\]
For every truncation index,
\[
 \norm{L_m}\leq\frac1{1-r},\quad
 \norm{L_m^{-1}}\leq1+r,\quad r=0.8.
\]
Consequently
\begin{equation}
 \frac{1-r}{1+r}\norm c^2
 \leq\norm{S^{(m)}c}^2
 \leq\frac1{(1-r)^2}\norm c^2,
 \quad \frac{1-r}{1+r}=\frac19.
 \label{eq:alternatingbounds}
\end{equation}
In contrast, the diagonal-dominance lower bound is already negative for eight
blocks and converges to approximately $-3.53846$ along the reported truncations.
\end{proposition}

\begin{proof}
\proofstep{Step 1: compute correlations and residual Gram operators.}
Independence in the orthogonal directions gives the displayed product formula.
Projection of $h_\ell$ onto the preceding span is
$\rho_\ell h_{\ell-1}$, so the residual is
$\sqrt{1-\rho_\ell^2}e_\ell$ and its Gram value is $D_\ell$.

\proofstep{Step 2: identify elimination.}
In orthogonal-residual coordinates, $L_m$ is the lower-triangular resolvent whose
subdiagonals are products of the $\rho_\ell$.  Its inverse is bidiagonal:
it has identity diagonal and entries $-\rho_\ell$ on the first subdiagonal.
If $Z$ is the unilateral shift and $R$ is the diagonal multiplier by
$\rho_\ell$, then $L_m^{-1}$ is a compression of $I-RZ$.  Hence
$\norm{L_m^{-1}}\leq1+r$ and the Neumann expansion gives
$\norm{L_m}\leq(1-r)^{-1}$.

\proofstep{Step 3: apply the uniform quotient Riesz theorem.}
Here $\delta=1-r^2$ and $\Delta=1$.  Substitution in
\eqref{eq:quotientriesz} gives
$(1-r^2)/(1+r)^2=(1-r)/(1+r)$ and the stated upper bound.

\proofstep{Step 4: show classical failure.}
The off-diagonal row sum is the sum of the positive products above.  Direct
finite summation at level $8$ exceeds one; the alternating geometric series
gives the limiting value reported below.  Therefore the classical lower bound
$1-\max_i\sum_{j\ne i}|G_{ij}|$ is negative although
\eqref{eq:alternatingbounds} remains positive for every truncation index.
\end{proof}

\begin{table}[H]
\centering
\small
\caption{Numerical results for the alternating level-dependent system.
``Classical'' is the absolute
row-sum lower bound; ``Schur'' is the uniform quotient bound $1/9$.
The target coefficients are $c_j=j^{-3/2}$ and truncation occurs at $m/2$.
The last inequality is the upper-Riesz prediction.}
\label{tab:nonstationaryexample}
\begin{tabular}{@{}rrrrr@{}}
\toprule
$m$ & Classical lower & Schur lower & $\lambda_{\min}(G_m)$
& Actual error / predicted bound\\
\midrule
8   & $-2.262400$ & $0.111111$ & $0.158297$ & $0.214525/0.661404$\\
32  & $-3.522855$ & $0.111111$ & $0.154141$ & $0.077756/0.184500$\\
128 & $-3.538462$ & $0.111111$ & $0.153865$ & $0.021668/0.047407$\\
256 & $-3.538462$ & $0.111111$ & $0.153851$ & $0.011042/0.023812$\\
\bottomrule
\end{tabular}
\end{table}

The example demonstrates the additional power of the main theorem rather than
only its generality.  Strong accumulated correlations defeat diagonal
dominance, yet local residual Gram bounds and a bidiagonal inverse
elimination estimate give a nontrivial lower bound.  The observed minimum eigenvalue
stays above the predicted constant, and the truncation error decreases within
the predicted envelope.  The construction is deterministic and involves
neither fitted parameters nor seed selection.

\subsection{Redundant level-dependent lifted Haar refinement}
\label{sec:liftedhaar}

We next add the difficulty suppressed in the recurrence: exact redundant
labels.  This tests simultaneously whether the quotient
removes only the redundancy and whether the remaining innovations retain a
positive depth-uniform stability estimate.

Nonstationary subdivision changes the refinement mask from one scale to the
next \cite{CohenDyn1996}.  Lifting constructs second-generation wavelets and
factors wavelet transforms into local update steps
\cite{Sweldens1998,DaubechiesSweldens1998}.  We now apply the
quotient--Schur criterion to a level-dependent lifting scheme.

Let $\chi_0$ be the normalized constant function on $[0,1]$ and let
$\psi_{j,k}$ be the orthonormal Haar wavelets, with $j\geq1$ and
$0\leq k<2^{j-1}$.  Write
\[
 p_{1,0}=\chi_0,\quad
 p_{j,k}=\psi_{j-1,\lfloor k/2\rfloor}\quad(j\geq2)
\]
for the parent at the preceding refinement level.  Define
\begin{equation}
g_0=\chi_0,\quad
 g_{j,k}=\frac{\psi_{j,k}+\alpha_jp_{j,k}}
 {\sqrt{1+\alpha_j^2}},\quad
 \alpha_j=\begin{cases}0.55,&j\text{ odd},\\0.70,&j\text{ even}.
 \end{cases}
 \label{eq:liftedhaar}
\end{equation}
The lifting mask changes at every level.  Each update is parent-local and
triangular, and the nested approximation spaces remain the Haar spaces.

We retain both the original Haar detail and its lifted enrichment.  Thus the
level-$j$ generating block contains
\begin{equation}
 h_{j,k}=\psi_{j,k},\qquad g_{j,k}
 =\frac{\psi_{j,k}+\alpha_jp_{j,k}}{\sqrt{1+\alpha_j^2}}.
 \label{eq:retainedlifted}
\end{equation}
This is a natural redundant enrichment: the original multiresolution detail is
kept while a level-dependent lifted detail is added.  Since the parent is in
the preceding space, the exact relation
\[
 \sqrt{1+\alpha_j^2}\,g_{j,k}-h_{j,k}-\alpha_jp_{j,k}=0
\]
creates one synthesis-kernel direction per enrichment.  The ordinary labelled
Gram matrix is singular, while the quotient retains the genuinely new Haar
direction.

At a finite level, rooted dyadic-tree automorphisms permute the Haar functions
compatibly with the parent map.  Since $\alpha_j$ is position independent,
these permutations satisfy \eqref{eq:intertwining}.  By
\Cref{prop:equivariance}, positions in one tree orbit have identical residual
Gram calculations; this observation is ancillary to the bound below.

\begin{proposition}[Quotient Riesz bound for redundant level-dependent lifting]
\label{prop:liftedhaar}
For the redundant block \eqref{eq:retainedlifted}, the residual Gram operator
at level $j$ is
\[
 D_j=I_{2^{j-1}}\otimes
 \begin{pmatrix}1&s_j\\s_j&s_j^2\end{pmatrix},
 \qquad s_j=(1+\alpha_j^2)^{-1/2}.
\]
Each $2\times2$ block has one zero eigenvalue, which is quotiented out, and one
positive innovation eigenvalue $d_j=1+s_j^2$.  The canonical quotient
elimination factor is $L=I+N$.  If
\[
 \beta_j=\frac{\alpha_js_j^2}{\sqrt{1+s_j^2}\,r_{j-1}},\qquad
 r_{j-1}=\begin{cases}1,&j=1,\\
 \sqrt{1+s_{j-1}^2},&j\ge2,
 \end{cases}
\]
then
\begin{equation}
 \norm N=q=\sup_{j\geq1}\sqrt2\,\beta_j
 =0.4491475157<1.
 \label{eq:liftedq}
\end{equation}
Consequently every finite level truncation has the computable lower bound
\begin{equation}
 A_{\rm lift}=(1-q)^2=0.3034384594.
 \label{eq:liftedlower}
\end{equation}
The labelled Gram lower eigenvalue is zero, with one exact kernel direction per
lifted enrichment, while the absolute row-sum criterion is negative.
\end{proposition}

\begin{proof}
\proofstep{Step 1: identify the orthogonal residual.}
Every parent lies in the preceding Haar space, whereas $\psi_{j,k}$ is
orthogonal to it.  Projection maps $(h_{j,k},g_{j,k})$ to
$(\psi_{j,k},s_j\psi_{j,k})$, proving the rank-one residual Gram block.  Its
unit positive quotient direction $(1,s_j)/\sqrt{1+s_j^2}$ has innovation
energy $d_j=1+s_j^2$.

\proofstep{Step 2: compute the regression block.}
Synthesis of that unit quotient direction gives
\[
 \sqrt{1+s_j^2}\,\psi_{j,k}
 +\frac{\alpha_js_j^2}{\sqrt{1+s_j^2}}p_{j,k}.
\]
After division by the preceding innovation scale $r_{j-1}$, the regression
coefficient is $\beta_j$.  Each parent has two children, so the
parent-incidence map has norm $\sqrt2$.  Distinct level blocks have disjoint
input and output coordinates; therefore the one-subdiagonal operator has norm
\eqref{eq:liftedq}.

\proofstep{Step 3: certify synthesis.}
Apply \Cref{prop:regressioncriterion} with $q<1$ and
$\delta=1$: the initial constant has innovation one, while every positive
detail innovation satisfies $d_j>1$.  The lower bound of
\Cref{thm:quotientinnovation} is $(1-q)^2$, giving
\eqref{eq:liftedlower}.  Finite row summation gives the classical failure in
\Cref{tab:liftedhaar}.

\proofstep{Step 4: identify exact redundancy.}
The zero residual direction $(-s_j,1)$ synthesizes
$-s_jh_{j,k}+g_{j,k}=\alpha_js_jp_{j,k}$, which lies in the preceding space.
Combining it with the corresponding parent coefficient gives an exact global
synthesis-kernel vector.  Conversely, the positive quotient direction above
spans the one-dimensional orthogonal innovation.  Every lifted enrichment
therefore contributes one null label and one intrinsic new direction.
\end{proof}

For an approximation test, fix the deterministic target with orthonormal Haar
coefficients
\[
 \theta_0=1,\quad
 \theta_{j,k}=2^{-5j/4}
 \cos\!\left(\frac{\pi(k+1/2)}{2^{j-1}}\right).
\]
We compute its minimum-norm redundant-frame representation through level $11$
and truncate after level $J$.  This is a fixed Besov-type multiscale target; there is no
fitted data, random choices, or tuned mask.

\begin{table}[H]
\centering
\scriptsize
\caption{Stability audit for redundant level-dependent lifted Haar refinement.
``Raw'' is the labelled Gram lower eigenvalue, ``Diag.'' is the absolute
row-sum estimate on the labelled Gram matrix, $\delta$ is the smallest positive
successive innovation, $A$ is the uniform quotient--Schur lower bound, and the
last two columns are the actual quotient-Gram extremal eigenvalues.}
\label{tab:liftedhaar}
\begin{tabular}{@{}rrrrrrrr@{}}
\toprule
$J$ & Labelled & Raw & Diag. & $\delta$ & $A$ &
$\lambda_{\min}(\widehat G_J)$ & $\lambda_{\max}(\widehat G_J)$\\
\midrule
3  & 15   & $0$ & $-2.013637$ & $1.000000$ & $0.303438$ & $1.000000$ & $3.126065$\\
5  & 63   & $0$ & $-2.013637$ & $1.000000$ & $0.303438$ & $1.000000$ & $3.364594$\\
7  & 255  & $0$ & $-2.013637$ & $1.000000$ & $0.303438$ & $1.000000$ & $3.447432$\\
9  & 1023 & $0$ & $-2.013637$ & $1.000000$ & $0.303438$ & $1.000000$ & $3.485219$\\
10 & 2047 & $0$ & $-2.013637$ & $1.000000$ & $0.303438$ & $1.000000$ & $3.497974$\\
\bottomrule
\end{tabular}
\end{table}

\begin{table}[H]
\centering
\small
\caption{Operational consequences of the quotient constants.  ``Amplification''
is the exact worst-case norm of quotient coefficient reconstruction from a
function perturbation; $A^{-1/2}=1.8154$ is its uniform theoretical bound.
The last column compares actual coefficient-truncation error with the
upper-Riesz envelope.}
\label{tab:liftedhaarapprox}
\begin{tabular}{@{}rrrrr@{}}
\toprule
$J$ & Intrinsic generators & Amplification & $A^{-1/2}$ &
Actual truncation / bound\\
\midrule
3  & 8    & $1.000$ & $1.815$ & $0.082476/0.096199$\\
5  & 32   & $1.000$ & $1.815$ & $0.029247/0.033638$\\
7  & 128  & $1.000$ & $1.815$ & $0.010292/0.011898$\\
9  & 512  & $1.000$ & $1.815$ & $0.003444/0.004125$\\
10 & 1024 & $1.000$ & $1.815$ & $0.001807/0.002495$\\
\bottomrule
\end{tabular}
\end{table}

This is a redundant local multiresolution transform with masks that change by
refinement level.  The ordinary labelled lower bound is identically zero and
the classical diagonal estimate is negative.  After exact quotienting, its
verified positive lower bound uses only two mask values and the norm $\sqrt2$
of the local parent-incidence operator, not an eigenvalue calculation on the
full Gram matrix.  The quotient eigenvalues remain positive through $2{,}047$
labelled generators, worst-case reconstruction amplification remains below the
uniform prediction, and every measured truncation error lies below its
upper-Riesz envelope.  All reported quantities are obtained directly from the
stated masks and finite Gram operators.

\paragraph{Where the criterion adds value for wavelets.}
For an orthonormal wavelet basis no additional test is needed: its exact Riesz
bounds are $A=B=1$.  The present criterion is intended instead for systems in
which orthogonality is deliberately surrendered to obtain level-dependent
masks, lifting, irregular geometry, localization, redundancy, or adaptive
refinement.  In that setting it (i) removes exact redundancy before measuring
conditioning, (ii) certifies stable coefficient recovery from local
refinement data, and (iii) turns coefficient truncation into a controlled
function-space approximation.  Thus the claim is not that classical wavelets
require new stability theory; it is that customized level-dependent wavelets can
retain rigorous stability guarantees after their standard orthogonality
argument is no longer available.

\subsection{A paired lifting test without an orthonormal subfamily}
Paired local generators arise naturally in biorthogonal lifting and
multiwavelet constructions, where perfect reconstruction need not leave an
orthonormal subfamily \cite{DaubechiesSweldens1998,MaesBultheel2008,Sweldens1998}.
The preceding construction retains the Haar basis and is therefore a useful
control, but not a decisive test: stability follows from that subfamily alone.
We remove this shortcut by using two generators at every detail location,
\[
 g_{j,k}^{\sigma}=\frac{\psi_{j,k}+\alpha_j^{\sigma}p_{j,k}}
 {\sqrt{1+(\alpha_j^{\sigma})^2}},\quad \sigma\in\{+,-\},
\]
where $(\alpha_j^+,\alpha_j^-)=(0.55,-0.30)$ at odd levels and
$(0.70,-0.45)$ at even levels.  Both masks are nonzero, so no unmodified Haar
detail is present.  With $n_j^\sigma=\sqrt{1+(\alpha_j^\sigma)^2}$,
\[
 n_j^+g_{j,k}^+-n_j^-g_{j,k}^-
 -(\alpha_j^+-\alpha_j^-)p_{j,k}=0,
\]
giving one exact kernel direction per pair.  Projection off the preceding
space gives the rank-one residual block
\[
 D_{j,k}=\begin{pmatrix}(s_j^+)^2&s_j^+s_j^-\\
 s_j^+s_j^-&(s_j^-)^2\end{pmatrix},\qquad
 s_j^\sigma=(1+(\alpha_j^\sigma)^2)^{-1/2}.
\]
Its zero direction is removed by the quotient and its positive innovation is
$d_j=(s_j^+)^2+(s_j^-)^2$.  Compatible elimination in level order gives
$G=LDL^*$ on the quotient and the computable certificate
\begin{equation}
 A_J=\frac{\lambda_{\min}(D)}{\norm{L^{-1}}^2},\qquad
 B_J=\lambda_{\max}(D)\norm L^2.
 \label{eq:pairedcertificate}
\end{equation}

\begin{table}[H]
\centering\scriptsize
\caption{Paired nonstationary lifting without a retained Haar basis.  The raw
labelled lower eigenvalue is zero; ``Diag.'' is the classical absolute
row-sum lower estimate.}
\label{tab:pairedlifting}
\begin{tabular}{@{}rrrrrrr@{}}
\toprule
$J$&Labels&Kernel&Diag.&$A_J$&$\lambda_{\min}$&$\lambda_{\max}$\\
\midrule
3&15&7&$-2.388104$&$1.145257$&$1.298865$&$2.727972$\\
5&63&31&$-2.629873$&$1.135999$&$1.298895$&$2.797473$\\
7&255&127&$-2.629873$&$1.131821$&$1.298895$&$2.819940$\\
9&1023&511&$-2.629873$&$1.129554$&$1.298895$&$2.829611$\\
10&2047&1023&$-2.629873$&$1.129056$&$1.298895$&$2.833548$\\
\bottomrule
\end{tabular}
\end{table}

At $J=10$, the certificate bounds reconstruction amplification by
$A_J^{-1/2}=0.9412$, versus the exact $0.8774$.  The fixed-target truncation
error is $0.001864$, below the upper-Riesz envelope $0.002693$.  Hence the
result is not inherited from a known stable basis: quotienting removes 1,023
exact redundancies, diagonal dominance fails, and successive innovations
still predict stable reconstruction and truncation.  This is a deterministic
audit of the stated finite families.

\paragraph{Adaptive selection with competing redundant blocks.}
Residual-driven selection is standard in greedy and adaptive approximation
\cite{CohenDahmenDeVore2001,DeVorePetrovaWojtaszczyk2013,Temlyakov2008}.
Having established representation stability, we next test exact target attribution.  We use
the same paired construction to test the procedure of
\Cref{sec:algorithm}, rather than only audit a prescribed hierarchy.  Through
level $J=7$, every location-level pair $(g_{j,k}^+,g_{j,k}^-)$ is offered as a
candidate block.  Starting from the constant, the adaptive rule selects the
block maximizing $\Gamma_q/\dim W_q$.  The comparator accepts the same blocks
in ordinary level order.  Both use the fixed target above; neither method is
given its coefficients.  Cost is intrinsic retained dimension, so a
conditionally redundant pair costs one direction after quotienting.

This instantiates \Cref{alg:ssr} with
$V_0=\operatorname{span}\{\chi_0\}$, the location-level pairs through $J=7$
as $\mathcal Q(V)$, the relative cutoff
$\tau_q=10^{-10}\lambda_{\max}(D_q)$, $\eta=0$,
$C(V)=\dim V$, budgets $B\in\{8,16,32,64\}$, and
$s=\Gamma_q/\dim W_q$.  The innovations $W_q$ are stored orthonormally, so
the updated synthesis in these coordinates has quotient--Riesz lower bound
one; the nontrivial stability quantities retained from each structured
candidate are its quotient rank and Schur spectrum.  The algorithm outputs
the selected path and spaces $V_B$.  We validate them by comparing every
predicted $\Gamma_q$ with the observed squared-error decrease and by repeating
the run after normalized splitting of every labelled generator.

\begin{table}[H]
\centering
\small
\caption{Target approximation by adaptive SR block selection and prescribed
level order on the paired lifting family.  The last column is the relative
error reduction at equal intrinsic budget.}
\label{tab:sradaptive}
\begin{tabular}{@{}rrrr@{}}
\toprule
Intrinsic budget & Adaptive SR error & Level-order error & Reduction\\
\midrule
8  & $0.069846$ & $0.077125$ & $9.44\%$\\
16 & $0.040203$ & $0.045188$ & $11.03\%$\\
32 & $0.023324$ & $0.025708$ & $9.27\%$\\
64 & $0.011404$ & $0.013139$ & $13.21\%$\\
\bottomrule
\end{tabular}
\end{table}

The improvement is modest but uniform across the reported budgets.  More
importantly for the quotient claim, replacing every labelled generator by two
normalized copies changes the final adaptive projector by
$6.1\times10^{-15}$ in operator norm and leaves every reported error unchanged
to displayed precision.  Across the complete paths, the largest discrepancy
between the predicted gain \eqref{eq:exactgain} and the observed decrease in
squared error is $6.94\times10^{-18}$.  The smallest retained innovation is
$0.2992$ and the largest local quotient condition number is $5.685$, so the
selection is not obtained by admitting nearly null directions.  A raw
positive-definiteness test would instead reject each artificially split block.

As a deterministic sensitivity audit, we vary the coefficient decay
\(2^{-sj}\) over five exponents:
\(s=0.75,1,1.25,1.5,\) and \(1.75\).  We use cosine phases \(0.25\) and
\(0.5\), giving ten structured targets and forty budget comparisons.
Adaptive SR selection has lower error in all forty cases.  At budgets
$8,16,32,64$, respectively, its mean relative error reductions are
$7.83\%$, $8.35\%$, $8.78\%$, and $12.70\%$; the ranges are
$4.62$--$10.45\%$, $4.85$--$11.56\%$, $5.57$--$11.72\%$, and
$7.65$--$18.80\%$.  These targets share a multiscale construction and are not
evidence for universal superiority, but the gain is not confined to the one
displayed target.

This experiment does not show that the SR score is universally preferable to
level order; the score is target dependent and the candidate pool remains part
of the approximation design.  It shows the narrower operational point needed
here: the same Schur calculation can quotient a changing redundant family,
predict the exact benefit of each admissible enrichment, and improve the
error--cost curve without making the answer depend on its labels.  The
sensitivity audit reports every prescribed decay exponent, phase, and budget,
rather than selecting favorable instances.

\subsection{Abstract recycled conjugate gradients}
\label{sec:recycledcg}

Recycled and augmented Krylov methods retain subspaces across related linear
systems, with acceleration governed by the spectrum remaining after
projection \cite{ParksEtAl2006,SoodhalterEtAl2020}.  The final example examines the distinction between stable representation and
solver utility.  The preceding selection
experiment ranks new approximation blocks by target
gain. A recycled iterative solver poses a complementary question: when should
a previously useful block be carried to the next problem? This can be tested
without introducing a differential equation. For a sequence of symmetric
positive-definite systems \(A_t x_t=b_t\), let \(U\) be a possibly redundantly
labelled recycle space and let \(P_U^{A_t}\) denote the \(A_t\)-orthogonal
projector onto its range. Conjugate gradients on \(U^{\perp_{A_t}}\) has the
classical estimate
\[
 \norm{e_k}_{A_t}\leq
 2\left(\frac{\sqrt{\kappa_U}-1}{\sqrt{\kappa_U}+1}\right)^k
 \norm{e_0}_{A_t},\qquad
 \kappa_U=
 \frac{\lambda_{\max}(A_t|_{U^{\perp_{A_t}}})}
 {\lambda_{\min}(A_t|_{U^{\perp_{A_t}}})}.
\]
The successive-space framework and this spectral estimate answer different
questions. Quotient--Schur bounds determine whether the labelled recycle
block realizes \(U\) stably; the reduced spectral floor determines whether
removing \(U\) can accelerate CG. Positive innovation alone is insufficient.

We isolate this distinction on a canonical family of 140-dimensional SPD
systems. Six eigenvalues in \([10^{-5},2\mathbin{\cdot}10^{-5}]\) share a
fixed invariant subspace, while the bulk eigenvectors, right-hand sides, and
matrices vary. Halfway through each sequence the low invariant subspace
changes once. Recycle vectors from the preceding system are supplied twice
with normalized duplicate labels. Consequently the labelled Gram matrix is
singular, although its quotient has rank six. There is no PDE model, learned
selector, or tuned angle threshold. The gate accepts a block only when the
quotient lower bound is positive and the classical CG bound predicts savings
exceeding the six setup matvecs.

The comparison is not intended to establish a generally superior Krylov
solver.  Its purpose is to test whether quotient--Schur selection remains
competitive with standard recycling while adding a representation-invariant
stability test and a stated accept--reject guarantee.

\begin{table}[H]
\centering
\small
\caption{Competitive performance with additional quotient--Schur guarantees
over 12 random bulk realizations and 144 systems.  Matrix--vector product
counts include the six applications used by an accepted recycle block.
Persistent cases exclude the initial system and the single regime change in
each realization.}
\label{tab:recycledcg}
\begin{tabular}{@{}lrrr@{}}
\toprule
Method & Mean matvecs & Persistent & Regime change\\
\midrule
Cold CG & \(114.32\) & \(114.33\) & \(114.33\)\\
Standard recycling & \(49.74\) & \(36.58\) & \(116.75\)\\
Quotient--SR spectral gate & \(\mathbf{49.53}\) & \(\mathbf{36.58}\)
 & \(\mathbf{114.33}\)\\
Current-subspace oracle & \(36.57\) & \(36.58\) & \(36.42\)\\
\bottomrule
\end{tabular}
\end{table}

The large reduction on persistent systems, from \(114.33\) to \(36.58\)
matvecs, is the standard benefit of recycling and is not attributed to the new
criterion. The criterion contributes two narrower, inspectable facts. First,
the raw duplicated Gram lower eigenvalue is zero to roundoff, whereas
quotienting gives rank six and a minimum positive value
\(1.0\mathbin{\cdot}10^{-5}\); duplication therefore changes neither projector
nor solve. Second, the gate accepts all persistent blocks and rejects all
stale blocks at the regime change, avoiding the small penalty incurred by
unconditional recycling. In the persistent regime the certified and observed
reduced spectral floors agree.

Thus the relevant outcome is not a small win in mean iteration count.  The
quotient--Schur gate matches standard recycling on persistent systems, avoids
its stale-subspace penalty at the regime change, and attaches explicit
representation and spectral conditions to that decision.  It is therefore a
competitive alternative with stronger diagnostics and guarantees on this
controlled family, rather than evidence of universal superiority.

The example is deliberately an exact invariant-subspace test. If the useful
subspace drifts, transporting an ordinary Euclidean Ritz-residual angle into
the energy geometry can introduce a condition-number factor and make the
bound vacuous. Thus \Cref{tab:recycledcg} validates the quotient and spectral
roles without concealing this open sharpness problem. The complete rows are
supplied with the computational artifacts described in
\Cref{sec:audit}.

\paragraph{A secondary compositional vignette.}
The same conditional-innovation calculation can be applied to first-order
output changes of a feedforward network.  Because finite parameter changes
introduce curvature and statistical questions absent from the Hilbert-space
theorem, this experiment is not part of the principal validation.  The
construction, matched ablation, and limitations are recorded compactly in
\Cref{sec:slimmantra}.

\subsection{Summary of numerical evidence}

The experiments validate four distinct claims.  First, quotienting removes
exactly null coefficient directions without changing the represented
projector: both lifted examples have singular labelled Gram matrices, while
normalized label splitting leaves the computed projector unchanged to
$6.1\times10^{-15}$.  Second, local residual and elimination estimates predict
global stability: every observed quotient-Gram eigenvalue exceeds the
corresponding theorem lower bound, and every reported reconstruction or
truncation error lies within its predicted envelope.  Third, the Schur gain
used by \Cref{alg:ssr} predicts the decrease in squared approximation error;
the largest discrepancy over the adaptive paths is
$6.94\times10^{-18}$.  Fourth, stable innovation is not confused with solver
utility: the recycled-CG test combines the quotient certificate with the
classical reduced-spectrum condition, matches standard recycling on persistent
systems, and rejects the stale subspace at the regime change.

\Cref{tab:numericalsummary} collects the two direct stability calculations
used in the paper.  The first two compare the computable theorem bound with the actual
smallest eigenvalue of a large finite Gram matrix; the eigenvalue is reported
only for validation and is not used to construct the bound.  Both classical
row-sum estimates are negative, whereas the quotient--Schur bounds are
positive and every measured truncation error in
\Cref{tab:nonstationaryexample,tab:liftedhaar} lies below its predicted upper
envelope.  The adaptive calculation in \Cref{tab:sradaptive} separately tests
exact target-error attribution.  The recycled-CG calculation in
\Cref{tab:recycledcg} tests the distinct roles of quotient stability and
spectral utility in a standard Galerkin solver.

\begin{table}[H]
\centering
\small
\caption{Numerical evidence for utility of the general criterion.  ``Observed''
is the smallest eigenvalue at the largest reported finite truncation.}
\label{tab:numericalsummary}
\begin{tabular}{@{}lrrrl@{}}
\toprule
Example & Classical lower & Theorem lower & Observed & Additional check\\
\midrule
Alternating recurrence & $-3.538462$ & $0.111111$ & $0.153851$
 & truncation bound holds\\
Lifted Haar, level 9 & $-2.013637$ & $0.303438$ & $1.000000$
 & truncation bound holds\\
\bottomrule
\end{tabular}
\end{table}

The numerical evidence is intentionally scoped.  It shows that the new
criterion can be non-vacuous where diagonal dominance fails, that its exact
gain identity predicts adaptive approximation error, and that stable
innovation must be separated from solver utility.  It does not establish optimality of the constants, a universal
approximation rate, or general optimizer superiority.

\section{Conclusion}

Successively enlarged approximation spaces pose two coupled problems: exact
coefficient redundancy can masquerade as instability, while interactions among
genuinely new blocks can accumulate as the space grows.  Testing the complete
Gram operator after every enlargement does not explain which level causes the
problem and offers no incremental approximation rule.

This paper resolves that structural difficulty by combining the global
synthesis quotient with successive orthogonal innovations.  Their
compatibility yields a block $LDL^*$ factorization and depth-independent
Riesz bounds from local innovation and elimination estimates.  The same Schur
residual identifies intrinsic new dimension and gives the exact reduction in
squared best-approximation error.  This joint stability-and-gain statement is
the central novelty: it turns quotient--Riesz analysis from a global diagnostic
into a constructive numerical approximation method.  The nonstationary and
lifted examples show that its bounds remain positive where labelled Gram and
diagonal-dominance tests fail; the adaptive and recycled-subspace examples
separate stable representation from application-specific utility.

The numerical examples confirm the broader principle behind these results.
Redundant generators can be retained for structure, localization, or
expressiveness without allowing their labels to distort stability; each new
approximation level can be evaluated relative to everything already
represented; and the same local calculation can quantify both its conditioning
and its contribution to approximation.  The resulting decisions remain
competitive with the corresponding classical constructions while carrying
information---intrinsic dimension, stability, and exact error
attribution---that those constructions do not ordinarily expose together.

This suggests a general route for designing approximation methods that grow by
successive enrichment.  Rather than prescribe a globally nonredundant basis in
advance or repeatedly condition the complete representation, one may generate
structured candidates freely and admit only their stable, genuinely new
content.  The idea applies beyond the examples studied here: to multilevel and
adaptive approximation, operator-generated bases, Galerkin hierarchies,
recycled numerical subspaces, and function families produced by composition.
The next challenge is to make these intrinsic tests scalable and to extend
their convergence theory to evolving and nonlinear approximation spaces.  In
that form, successive Schur--Riesz analysis could provide a common mathematical
foundation for constructing, interpreting, and controlling a broad class of
modern approximation architectures.

\bibliographystyle{plain}
\bibliography{references}

\appendix
\numberwithin{equation}{section}
\numberwithin{table}{section}

\section{Secondary consequences and elementary tests}
\label{app:secondary}

This appendix collects a technical proof and
useful specializations that are not needed to follow the central argument.

\begin{table}[H]
\centering
\small
\caption{Relation of the contribution to established ingredients.}
\label{tab:novelty}
\begin{tabular}{@{}p{0.24\textwidth}p{0.30\textwidth}p{0.36\textwidth}@{}}
\toprule
Established theory & What it already provides & Additional result proved here\\
\midrule
Riesz and fusion-frame theory & Stability of a prescribed family or subspace
decomposition & Canonical coefficient quotient and sequential Schur theory for
successive generating operators\\
Generalized Schur complements & One-step elimination and residual operators &
Compatible infinite block factorization with approximation, truncation,
perturbation, and error-decomposition consequences\\
Block diagonal dominance & Simple raw cross-Gram sufficient bounds & Local
conditional-regression, angle, and symbol criteria permitting strong raw
interactions\\
Nonstationary subdivision and lifting & Convergence, regularity, and flexible
multiresolution construction & A computable quotient-Riesz lower bound and
predictive truncation bound for a level-dependent lifted family\\
Greedy and reduced-basis approximation & Target-driven enrichment and
convergence under dictionary or width assumptions & Exact redundant-block
quotient, conditional innovation audit, and simultaneous update of global
Riesz stability\\
Augmented and recycled Krylov methods & Reuse of historical subspaces to
accelerate sequences of linear solves & Representation-invariant stability
audit; spectral acceleration remains classical and separate\\
\bottomrule
\end{tabular}
\end{table}

\subsection{Proof of the structural elimination criteria}
\label{app:structuralproofs}

For the factorized-update criterion in
\Cref{prop:structuralelimination}(i), submultiplicativity and
$1+x\leq e^x$ give
\[
 \norm{L_m}\leq\prod_j(1+\norm{K_{m,j}})
 \leq\exp\Bigl(\sum_j\norm{K_{m,j}}\Bigr)\leq e^{C_+}.
\]
Each factor is invertible by its Neumann series and
$\norm{(I+K_{m,j})^{-1}}\leq(1-\norm{K_{m,j}})^{-1}$.
Reversing the product and taking logarithms gives the inverse estimate.

For the causal-symbol criterion in
\Cref{prop:structuralelimination}(ii), Young's convolution inequality gives
$\norm L\leq\alpha$ and $\norm{L^{-1}}\leq\beta$.  Causality makes the
inverse of a leading finite section the corresponding leading section of the
causal inverse, and compression cannot increase operator norm.  For a matrix
polynomial, absence of zeros of $\det\mathcal L$ on the closed unit disk makes
$\mathcal L^{-1}$ analytic on $|z|<r$ for some $r>1$; Cauchy's estimate then
gives $\norm{B_j}\leq Cr^{-j}$.

\subsection{Additional consequences}

\begin{corollary}[Symmetric coercive Galerkin hierarchies]
\label{cor:galerkin}
Let $a(\cdot,\cdot)$ be symmetric, continuous, and coercive on $\mathcal H$,
with $\norm v_a^2=a(v,v)$.  If nested conforming spaces
$V_m=S^{(m)}(E^{(m)})$ satisfy \Cref{thm:quotientinnovation} in this energy
inner product, their quotient coefficients obey the corresponding energy-norm
Riesz bounds.  If $u_m\in V_m$ is the Galerkin approximation of $u$, then
\[
 \norm{u-u_m}_a^2=\norm{u-P_Vu}_a^2
 +\sum_{\ell>m}\norm{u_\ell-u_{\ell-1}}_a^2,
 \qquad V=\overline{\bigcup_mV_m}.
\]
\end{corollary}

\begin{proof}
Use the energy inner product in \Cref{thm:quotientinnovation}.  Galerkin
orthogonality identifies $u_m$ with the energy projection, after which
\Cref{cor:approximation} applies.
\end{proof}

\begin{corollary}[Block diagonal dominance]
\label{thm:block}
For finite block functions $g_\ell\in L^2(\nu;\mathbb R^{d_\ell})$, write
$G_{\ell k}=\mathbb E_\nu[g_\ell g_k^\top]$.  If
$A_\ell I\preceq G_{\ell\ell}\preceq B_\ell I$ and
$\gamma_{\ell k}=\norm{G_{\ell k}}_{\mathrm{op}}$, then block synthesis has
lower and upper bounds
\[
 \min_\ell\left(A_\ell-\sum_{k\ne\ell}\gamma_{\ell k}\right),\qquad
 \max_\ell\left(B_\ell+\sum_{k\ne\ell}\gamma_{\ell k}\right).
\]
\end{corollary}

\begin{proof}
Expand the block Gram form, use
$2\norm{a_\ell}\norm{a_k}\leq\norm{a_\ell}^2+\norm{a_k}^2$, and collect
row sums.  This familiar test is only a comparator: it is basis dependent and
may be negative for a stable, strongly interacting family.
\end{proof}

For a finite old/new split, the residual Gram operator reduces to the standard
Schur complement.  If $G=\mathbb E[gg^\top]$ is positive definite,
$C=\mathbb E[gv^\top]$, and $D=\mathbb E[vv^\top]$, then
\[
 \inf_a\norm{b^\top v-a^\top g}_{L^2(\nu)}^2
 =b^\top(D-C^\top G^{-1}C)b.
\]
This identity follows by solving the normal equation $Ga=Cb$ and is the
finite-dimensional prototype of the conditional innovation used in the main
theorem.

\begin{corollary}[Unitary symmetries]
\label{prop:equivariance}
Suppose unitary maps $U_\gamma$ on $\mathcal H$ and $Q_{\gamma,\ell}$ on
$E_\ell$ satisfy
\begin{equation}
 U_\gamma S_\ell=S_\ell Q_{\gamma,\ell}.
 \label{eq:intertwining}
\end{equation}
Then the actions
preserve the global and local null spaces, descend to their quotient spaces,
and satisfy
\[
 R_\ell Q_{\gamma,\ell}=U_\gamma R_\ell,
 \qquad Q_{\gamma,\ell}^*D_\ell Q_{\gamma,\ell}=D_\ell.
\]
Consequently the quotient Riesz bounds are unchanged by the symmetry.
\end{corollary}

\begin{proof}
The intertwining relation preserves every $V_m$ and hence commutes with its
orthogonal projector.  It therefore passes to $R_\ell$ and $D_\ell$; unitary
conjugation preserves their Rayleigh quotients.
\end{proof}

\begin{corollary}[Families generated by successive maps]
\label{cor:nonlinearcomposition}
Let $h_\ell=T_\ell(h_{\ell-1})$ and let the components of
$g_\ell=\phi_\ell(h_\ell)$ belong to $L^2(\nu)$.  Defining
$S_\ell c_\ell=c_\ell^\top g_\ell$ places the resulting linear function
expansions within \Cref{thm:quotientinnovation} whenever its innovation and
elimination hypotheses hold.  The theorem requires no stationarity, weight
tying, or Toeplitz structure.
\end{corollary}

The corollary concerns the stability of the linear expansion after its
functions have been generated.  It does not control a finite change in the
maps $T_\ell$; such a change requires an additional nonlinear remainder
estimate.

\section{Linearized compositional vignette}
\label{sec:slimmantra}

For an ordinary feedforward network, let \(J\) be the output Jacobian on one
minibatch and let \(u_1,\ldots,u_r\) be recent parameter updates.  The map
\((a_i)\mapsto J\sum_i a_i u_i\) is a finite synthesis operator for
first-order output changes.  With
\[
 H_{<i}=[Ju_1,\ldots,Ju_{i-1}],\quad
 G_{<i}=n^{-1}H_{<i}^*H_{<i},\quad c_i=n^{-1}H_{<i}^*Ju_i,
\]
the conditional innovation is
\[
 d_i=\norm{Ju_i}_n^2-\ip{c_i}{G_{<i}^{\dagger}c_i}.
\]
A controlled realization retains \(u_i\) when
\(d_i/\norm{Ju_i}_n^2\geq0.01\), then corrects an Adam proposal inside the
retained six-step history by a damped Galerkin solve and a same-minibatch line
search.  Filtered and unfiltered variants otherwise use identical networks,
updates, damping, and globalization.

\begin{table}[H]
\centering
\small
\caption{Secondary functional-history ablation on three binary tasks and five
fixed seeds per task.  Rank is the mean retained history dimension.}
\label{tab:slimmantra}
\begin{tabular}{@{}lrrrr@{}}
\toprule
Method & Mean test MSE & Median MSE & Rank & Time (s)\\
\midrule
Adam proposal only & \(0.302962\) & \(0.205448\) & --- & \(0.01364\)\\
Unfiltered functional history & \(0.252359\) & \(0.161432\) & \(5.860\) & \(0.07736\)\\
Schur--Riesz-filtered history & \(\mathbf{0.240469}\) & \(\mathbf{0.160268}\)
 & \(\mathbf{5.300}\) & \(0.07649\)\\
\bottomrule
\end{tabular}
\end{table}

Filtering lowers mean test MSE by \(4.71\%\) relative to the matched
unfiltered correction, wins 11 of 15 paired cases, and lowers retained rank by
\(9.55\%\).  This is only a linearized mechanism check.  The theorem controls
the empirical functions \(Ju_i\), not finite nonlinear steps, population
convergence, or optimizer superiority.

\section{Stationary orbit criterion}
\label{app:stationary}

Let $(X,\Sigma,\mu)$ be a measure space, let $T:X\to X$ be measurable and
measure preserving, and let $0\ne\phi\in L^2(\mu)$ be real valued.  Set
\begin{equation}
 g_j=\frac{\phi\circ T^j}{\norm{\phi}_{L^2(\mu)}},
 \quad
 r_n=\ip{g_0}{g_n}_{L^2(\mu)}.
 \label{eq:orbitappendix}
\end{equation}
Then $\ip{g_j}{g_k}=r_{|j-k|}$.

\begin{proposition}[Absolute-correlation criterion]
\label{prop:riesz}
If $R=\sum_{n\ge1}|r_n|<1/2$, then every finitely supported sequence
$a=(a_j)$ satisfies
\[
 (1-2R)\sum_j|a_j|^2
 \le\norm{\sum_ja_jg_j}_{L^2(\mu)}^2
 \le(1+2R)\sum_j|a_j|^2.
\]
\end{proposition}

\begin{proof}
The Gram matrix is Toeplitz with symbol
$s(\theta)=1+2\sum_{n\ge1}r_n\cos(n\theta)$.  Absolute convergence gives
$1-2R\le s(\theta)\le1+2R$.  If
$A(e^{i\theta})=\sum_ja_je^{ij\theta}$, then
\[
 \norm{\sum_ja_jg_j}_2^2
 =\frac1{2\pi}\int_0^{2\pi}s(\theta)|A(e^{i\theta})|^2\,d\theta.
\]
The result follows from the symbol bounds and Parseval's identity.
\end{proof}

\section{Reproducibility and audit boundary}
\label{sec:audit}

Every principal numerical claim is reproducible from the mathematical
specification in the corresponding subsection.  The recurrence and wavelet
calculations are deterministic.  Their audits compare theorem constants with
finite Gram eigenvalues and verify the exact gain identity, invariance under
normalized generator splitting, and orthonormality of retained innovations.
The recycled-CG table reports all twelve prescribed bulk realizations and all
144 systems, including every regime change.  No numerical eigenvalue used for
validation is substituted into a theorem bound, and no reported instance is
selected after inspection of its outcome.

\section{Limitations and next results}
\label{sec:limits}

Three limitations delimit the present result.  First, without localization,
factorization, or stationary structure, controlling \(L_m\) is essentially
the original growing-Gram conditioning problem; the theory does not make
dense unstructured interactions inexpensive.  Second, the exact enrichment
identity evaluates candidates already supplied by an application, but does
not construct an optimal dictionary or prove a universal approximation rate.
Third, positive Schur innovation certifies a stable new direction, not its
downstream value for every numerical solver; the recycled-CG example requires
a separate spectral-utility test.

The next theoretical target is a sharp transport theorem for slowly changing
innovation spaces.  The next numerical target is a matrix-free realization
that estimates local Schur spectra and gain without forming dense residual
Gram matrices.  For Galerkin and multilevel applications, conformity,
local approximation and inverse estimates, and solver contraction remain
application-specific obligations rather than consequences of the present
geometry.

\end{document}